\documentclass[a4paper]{amsart}
\usepackage{amssymb}

\PII{}
\makeatletter
\def\@setcopyright{\@empty}
\makeatother

\newcommand{\prn}[1]{\left(#1\right)}
\newcommand{\brc}[1]{\left\{#1\right\}}
\newcommand{\allp}{1\le p\le\infty}
\newcommand{\Lp}{L_{p,\alpha,\beta}}
\newcommand{\Lmu}{L_{1,2,2}}
\newcommand{\norm}[1]{\left\|#1\right\|_{p,\alpha,\beta}}
\newcommand{\normpar}[2]{\left\|#1\right\|_{#2}}
\newcommand{\E}{E_n(f)_{p,\alpha,\beta}}
\newcommand{\Epar}[2]{E_{#1}\left(#2\right)_{p,\alpha,\beta}}
\newcommand{\T}[3]{T_{#1}^{#2}\left(#3\right)}
\newcommand{\hatT}[3]{\hat T_{#1}^{#2}\left(#3\right)}
\newcommand{\Si}[1]{\left(1-#1^2\right)}
\newcommand{\Dl}[3]{\Delta_{#1}^{#2}\left(#3\right)}
\newcommand{\arr}[2]{{{#1}_1,\dots,{#1}_{#2}}}
\newcommand{\w}{\hat\omega_r(f,\delta)_{p,\alpha,\beta}}

\newcommand{\Px}[1]{P_{#1}^{(2,2)}}
\newcommand{\Py}[1]{P_{#1+2}^{(0,0)}}
\newcommand{\krn}[1]{%
  \left(\frac{\sin\frac{m#1}2}{\sin\frac{#1}2}\right)^{2q+4}}
\newcommand{\numericset}[1]{\mathbb #1}
\newcommand{\numN}{\numericset N}

\newtheorem{thm}{Theorem}[subsection]
\newtheorem{lmm}{Lemma}[subsection]
\newtheorem{cor}{Corollary}[subsection]

\newcounter{const}

\numberwithin{const}{thm}
\numberwithin{const}{lmm}
\numberwithin{const}{cor}

\makeatletter
\newcommand{\Cn}[1][]{%
  \stepcounter{const}C_{\theconst}%
  \@ifnotempty{#1}{\newcounter{#1}\setcounter{#1}{\arabic{const}}}}
\makeatother

\newcommand{\lastC}{C_{\theconst}}
\newcommand{\prevC}[1][1]{%
	{\countdef\n=255
	 \n=\theconst
	 \advance\n by-#1
	 C_{\number\n}}}

\numberwithin{equation}{subsection}

\renewcommand{\theconst}{\arabic{const}}

\begin{document}

\title[Approximation of classes of functions\dots]
	{Approximation of classes of functions defined by
	  a generalized $r$-th modulus of smoothness}
\author{M.~K.\ Potapov}
\address{M.~K.\ Potapov\\
	Department of Mechanics and Mathematics\\
	Moscow State University\\
	Moscow 117234\\
	Russia}
\author{F.~M.\ Berisha}
\address{F.~M.\ Berisha\\
	Faculty of Mathematics and Sciences\\
	University of Prishtina\\
	N\"ena Terez\"e~5\\
	38000 Prishtin\"e\\
	Kosova}
	\email{faton.berisha@uni-pr.edu}

\keywords{Generalised modulus of smoothness,
	asymmetric operator of generalised translation,
	coincidence of classes,
	best approximations by algebraic polynomials}
\subjclass{Primary 41A35, Secondary 41A50, 42A16.}
\date{}
\dedicatory{Dedicated to Proffesor P.~L.\ Ul'yanov
  on the occasion of his 70-th birthday}
\thanks{This work was done under the support of
  the Russian Foundation for Fundamental Scientific Research,
  Grant \#97-01-00010 and Grant \#96/97-15-96073.}

\begin{abstract}
In this paper, a $k$-th generalized modulus of smoothness
is defined based on an asymmetric operator of generalized
translation and a theorem is proved about
the coincidence of class of functions defined by this
modulus and a class of functions having given order
of best approximation by algebraic polynomials.
\end{abstract}

\maketitle

\textbf{Introduction. }%
In paper~\cite{potapov:mat-99},
an asymmetric operator of generalized translation was introduced
and by means of it
the corresponding generalized modulus of smoothness
of first order was defined.
Then a theorem was proved about coincidence
of the class of functions defined by this modulus
with the class of functions having a given order of best approximation
by algebraic polynomials.

In the present paper,
analogous results are obtained
for the generalized modulus of smoothness of order~$r$.
In addition,
the space in which the theorem of coincidence
of the corresponding classes of functions holds true is widened.

\subsection{}

For $1\le p<\infty$, as usual,
$L_p$~denotes the set all measurable functions~$f$ on~$[-1,1]$
for which
\begin{displaymath}
	\normpar f p=\prn{\int_{-1}^1|f(x)|^p\,dx}^{1/p}<\infty.
\end{displaymath}
For $p=\infty$,
$L_\infty$~is the space of all continuos functions~$f$ on~$[-1,1]$
with a norm
\begin{displaymath}
	\normpar f\infty=\max_{-1\le x\le1}|f(x)|.
\end{displaymath}

Denote by~$\Lp$ the set of functions~$f$ such that
$f(x)\*(1-x)^\alpha\*(1+x)^\beta\in L_p$,
and set
\begin{displaymath}
	\norm f=\normpar{f(x)(1-x)^\alpha(1+x)^\beta}p.
\end{displaymath}

By~$\E$ we denote the best approximation of $f\in\Lp$
by algebraic polynomials of degree not greater than~$n-1$
in~$\Lp$ metrics, that is,
\begin{displaymath}
	\E=\inf_{P_n\in\mathcal P_n}\norm{f-P_n},
\end{displaymath}
where~$\mathcal P_n$ is the set of algebraic polynomials
of degree not greater than~$n-1$.

By~$E(p,\alpha,\beta,\lambda)$ we denote the class of functions $f\in\Lp$
satisfying the condition
\begin{displaymath}
	\E\le Cn^{-\lambda},
\end{displaymath}
where $\lambda>0$ and~$C$ is a constant
not depending on~$n$ $(n\in\numN)$.

For functions~$f$ we define
the operator of \emph{generalized translation} $\hatT t{}{f,x}$
by
\begin{multline*}
	\hatT t{}{f,x}=\frac1{\pi\Si x}
	  \int_0^\pi
		\bigg(
		  1-\prn{x\cos t-\sqrt{1-x^2}\sin t\cos\varphi}^2
		  -2\sin^2t\sin^2\varphi\\
	+4\Si{x}\sin^2t\sin^4\varphi
		\bigg)
		  f(x\cos t-\sqrt{1-x^2}\sin t\cos\varphi)\,d\varphi.
\end{multline*}

By means of this operator of generalized translation
we define the \emph{generalized difference}
of order~$r$ by
\begin{align*}
	\Dl t1{f,x}
	  &=\Dl t{}{f,x}=\hatT t{}{f,x}-f(x),\\
	\Dl{\arr t r}r{f,x}
	  &=\Dl{t_r}{}{\Dl{\arr t{r-1}}{r-1}{f,x},x}
	  \quad(r=2,3,\dotsc),
\end{align*}
and the \emph{generalized modulus of smoothness}
of order~$r$ by
\begin{displaymath}
	\w=\sup_{|t_i|_{i=1,\dots,r}\le\delta}
	  \norm{\Dl{\arr t r}r{f,x}}
	\quad(r=1,2,\dotsc).
\end{displaymath}

Consider the class $H(p,\alpha,\beta,r,\lambda)$
of functions $f\in\Lp$ satisfying the condition
\begin{displaymath}
	\w\le C\delta^\lambda,
\end{displaymath}
where $\lambda>0$ and~$C$ is a constant
not depending on~$\delta$.

The aim of the present paper
is to prove the following statement

\begin{thm}\label{th:coincidence}
	Let~$p$, $\alpha$, $\beta$ and~$r$ be given numbers
	such that $\allp$, $r\in\numN$;
	\begin{alignat*}3
		\frac12      &<\alpha\le2,
		  &\quad \frac12      &<\beta\le2
			  &\quad &\text{for $p=1$},\\
		1-\frac1{2p} &<\alpha<3-\frac1p,
		  &\quad 1-\frac1{2p} &<\beta<3-\frac1p
			  &\quad &\text{for $1<p<\infty$},\\
		1            &\le\alpha<3,
		  &\quad 1            &\le\beta<3
			  &\quad &\text{for $p=\infty$}.
	\end{alignat*}
	Then, for any~$\lambda$ satisfying the condition
	\begin{displaymath}
		\lambda_0=2\max\prn{
		  |\alpha-\beta|,\alpha-\frac32+\frac1{2p},
		  \beta-\frac32+\frac1{2p}}
		<\lambda<2r
	\end{displaymath}
	the class~$H(p,\alpha,\beta,r,\lambda)$
	coincides with the class~$E(p,\alpha,\beta,\lambda)$.
\end{thm}

The validity of Theorem~\ref{th:coincidence}
will follow from the validity of Theorems~\ref{th:HsubE}
and~\ref{th:EsubH},
which we are going to prove below.

\subsection{}

Put $y=\cos t$, $z=-\cos\varphi$ in the definition of $\hatT t{}{f,x}$
and denote the resulting operator by $\T y{}{f,x}$.
Let us rewrite it in the form
\begin{multline*}
	\T y{}{f,x}
	  =\frac1{\pi\Si x}\int_{-1}^1
		  \Big(
			1-R^2-2\Si y\Si z\\
	+4\Si x\Si y\Si{z}^2
		  \Big)
		  f(R)\frac{dz}{\sqrt{1-z^2}},
\end{multline*}
where $R=xy-z\sqrt{1-x^2}\sqrt{1-y^2}$.
We define the operator of generalized translation
of order~$r$ by
\begin{align*}
	\T y1{f,x}         &=\T y{}{f,x},\\
	\T{\arr y r}r{f,x} &=\T{y_r}{}{\T{\arr y{r-1}}{r-1}{f,x},x}
	  \quad(r=2,3,\dotsc).
\end{align*}

By $P_\nu^{(\alpha,\beta)}(x)$ $(\nu=0,1,\dotsc)$
we denote the Jacobi polynomials,
i.e.\ algebraic polynomials of degree~$\nu$
orthogonal on the segment~$[-1,1]$
with a weight $(1-x)^{\alpha}(1+x)^{\beta}$
and normalized by the condition
$P_\nu^{(\alpha,\beta)}(1)=1$ $(\nu=0,1,\dotsc)$.

For any integrable function~$f$ on~$[-1,1]$
with a weight $\Si{x}^2$,
we denote by $a_n(f)$
the Fourier--Jacobi coefficients of~$f$
with respect to the system of Jacobi polynomials
$\brc{\Px n(x)}_{n=0}^\infty$, i.e.
\begin{displaymath}
	a_n(f)=\int_{-1}^1f(x)\Px n(x)\Si{x}^2\,dx
	\quad(n=0,1,\dotsc).
\end{displaymath}

Introduce certain operators
which will play an auxiliary role later on.
First we set
\begin{align*}
	\T{1;y}{}{f,x}
	  &=\frac1{\pi\Si x}\int_{-1}^1\prn{1-R^2-2\Si y\Si z}f(R)
		  \frac{dz}{\sqrt{1-z^2}},\\
	\T{2;y}{}{f,x}
	  &=\frac8{3\pi}\int_{-1}^1\Si{z}^2f(R)\frac{dz}{\sqrt{1-z^2}},
\end{align*}
where $R=xy-z\sqrt{1-x^2}\sqrt{1-y^2}$,
and then define the corresponding operators of order~$r$
by
\begin{align*}
	\T{k;y}1{f,x}
	  &=\T{k;y}{}{f,x},\\
	\T{k;\arr y r}r{f,x}
	  &=\T{k;y_r}{}{\T{k;\arr y{r-1}}{r-1}{f,x},x}
	  \quad(r=2,3,\dotsc)
\end{align*}
for $k=1,2$.

\subsection{}

\begin{lmm}\label{lm:rho-sigma}
	Let $f\in\Lp$ and let the numbers~$p$, $\alpha$, $\beta$,
	$\rho$, $\sigma$ and~$\lambda$
	be such that
	$\allp$, $\rho\ge0$, $\sigma\ge0$,
	$\lambda>\lambda_0=2\max\{\rho,\sigma\}$;
	\begin{alignat*}3
		\alpha &>-\frac1p, &\quad \beta &>-\frac1p &\quad
			&\text{for $1\le p<\infty$},\\
		\alpha &\ge0, 		 &\quad \beta &\ge0 		 &\quad
			&\text{for $p=\infty$}.
	\end{alignat*}
	If there exists a sequence of algebraic polynomials
	$\brc{P_{2^n}(x)}_{n=0}^\infty$
	such that
		\begin{displaymath}
		\|f-P_{2^n}\|_{p,\alpha+\rho,\beta+\sigma}
		\le\frac{\Cn}{2^{n\lambda}},
	\end{displaymath}
	then the following inequalities also hold true
	\begin{displaymath}
		\norm{f-P_{2^n}}
		\le\frac{\Cn}{2^{n(\lambda-\lambda_0)}}
		\quad(n=1,2,\dotsc),
	\end{displaymath}
	where the constants~$\prevC$ and~$\lastC$
	do not depend on~$n$.
\end{lmm}

Lemma~\ref{lm:rho-sigma}
was proved in~\cite{potapov:trudy-75}.

\begin{lmm}\label{lm:bernshtein-markov}
	Let $P_n(x)$ be an algebraic polynomial
	of degree not greater than $n-1$,
	$\allp$, $\rho\ge0$, $\sigma\ge0$.
	Assume that
	\begin{alignat*}3
		\alpha &>-\frac1p, &\quad \beta &>-\frac1p &\quad
			&\text{for $1\le p<\infty$},\\
		\alpha &\ge0, 		 &\quad \beta &\ge0 		 &\quad
			&\text{for $p=\infty$}.
	\end{alignat*}
	Then
	\begin{gather*}
		\normpar{P'_n(x)}{p,\alpha+\frac12,\beta+\frac12}
		  \le\Cn n\norm{P_n},\\
		\norm{P_n}
		  \le\Cn n^{2\max\prn{\rho,\sigma}}
			\normpar{P_n}{p,\alpha+\rho,\beta+\sigma},
	\end{gather*}
	where the constants~$\prevC$ and~$\lastC$
	do not depend on~$n$.
\end{lmm}

Lemma was proved in~\cite{halilova:izv-74}.

\begin{lmm}\label{lm:T*P}
	The operators~$T_{1;y}$ and~$T_{2;y}$
	have the following properties
	\begin{gather*}
		\T{1;y}{}{\Px\nu,x}=\Px\nu(x)\Py\nu(y),\\
		\T{2;y}{}{\Px\nu,x}=\Px\nu(x)\Px\nu(y)
	\end{gather*}
	for $\nu=0,1,\dotsc$
\end{lmm}

Lemma~\ref{lm:T*P} was proved in~\cite{potapov:mat-99}.

\begin{lmm}\label{lm:T*fg}
	Let $g(x)\T{k;y}{}{f,x}\in\Lmu$ for every~$y$.
	Then for $k=1,2$ the following equality holds true
	\begin{displaymath}
		\int_{-1}^1f(x)\T{k;y}{}{g,x}\Si{x}^2\,dx
		=\int_{-1}^1g(x)\T{k;y}{}{f,x}\Si{x}^2\,dx.
	\end{displaymath}
\end{lmm}

\begin{proof}
	Let $k=1$ and
	\begin{multline*}
		I_1:=\int_{-1}^1f(x)\T{1;y}{}{g,x}\Si{x}^2\,dx\\
		=\frac1\pi\int_{-1}^1\int_{-1}^1f(x)g(R)\prn{1-R^2-2\Si y\Si z}
			  \Si x\frac{dz\,dx}{\sqrt{1-z^2}},
	\end{multline*}
	where $R=xy-z\sqrt{1-x^2}\sqrt{1-y^2}$.
	Performing change of variables in the double integral
	by the formulas
	\begin{equation}\label{eq:x-z}
		\begin{aligned}
		x&=Ry+V\sqrt{1-R^2}\sqrt{1-y^2},\\
		z&=-\frac{R\sqrt{1-y^2}-Vy\sqrt{1-R^2}}
			  {\sqrt{1-\prn{Ry+V\sqrt{1-R^2}\sqrt{1-y^2}}^2}},
		\end{aligned}
	\end{equation}
	we get
	\begin{multline*}
		I_1=\frac1\pi\int_{-1}^1\int_{-1}^1\Si Rf
			\prn{Ry+V\sqrt{1-R^2}\sqrt{1-y^2}}g(R)\\
		\times\prn{1-\prn{Ry+V\sqrt{1-R^2}\sqrt{1-y^2}}^2-2\Si y\Si V}
				\frac{dV\,dR}{\sqrt{1-V^2}}\\
		=\int_{-1}^1g(R)\T{1;y}{}{f,R}\Si{R}^2\,dR,
	\end{multline*}
	which proves the equality of the lemma for $k=1$.
	
	Let $k=2$ and
	\begin{multline*}
		I_2:=\int_{-1}^1f(x)\T{2;y}{}{g,x}\Si{x}^2\,dx\\
		=\frac8{3\pi}\int_{-1}^1\int_{-1}^1f(x)g(R)\Si{x}^2\Si{z}^2
			\frac{dz\,dx}{\sqrt{1-z^2}}.
	\end{multline*}
	Performing again the change~\eqref{eq:x-z}
	in the double integral
	we get
	\begin{multline*}
		I_2=\frac8{3\pi}\int_{-1}^1\int_{-1}^1f
			\prn{Ry+V\sqrt{1-R^2}\sqrt{1-y^2}}g(R)\Si{R}^2\\
		\times\Si{V}^2\frac{dV\,dR}{\sqrt{1-V^2}}
		  =\int_{-1}^1g(R)\T{2;y}{}{f,R}\Si{R}^2\,dR.
	\end{multline*}
	
	Lemma~\ref{lm:T*fg} is proved.
\end{proof}

\begin{cor}\label{cr:T*fg}
	If $f\in\Lmu$,
	then for every $r\in\numN$
	we have $\T{k;\arr r y}{r}{f,x}\in\Lmu$ $(k=1,2)$.
\end{cor}

\begin{proof}
	Put $g(x)\equiv1$ on $[-1,1]$.
	Taking into account that by Lemma~\ref{lm:T*P}
	\begin{gather*}
		\T{1;y}{}{1,x}=\T{1;y}{}{\Px0,x}
		  =\Px0(x)P_2^{(0,0)}(y)=\frac32 y^2-\frac12,\\
		\T{2;y}{}{1,x}=1,
	\end{gather*}
	we clearly have $f(x)\T{k;y}{}{1,x}\in\Lmu$ $(k=1,2)$.
	Hence, applying Lemma~\ref{lm:T*fg} we derive the relation
	\begin{displaymath}
		\int_{-1}^1\T{k;y}{}{f,x}\Si{x}^2\,dx
		=\int_{-1}^1f(x)\T{k;y}{}{1,x}\Si{x}^2\,dx
		\quad (k=1,2),
	\end{displaymath}
	which implies that $\T{k;y}{}{f,x}\in\Lmu$.
	Now the corollary can be proved by induction.
\end{proof}

\begin{lmm}\label{lm:T*P1}
	Let~$f$ be an integrable function on~$[-1,1]$
	with a weight~$\Si{x}^2$.
	For every natural number~$n$
	the following equality holds true
	\begin{displaymath}
		\int_{-1}^1\T{1;y}{}{f,x}P_n^{(1,1)}(y)\,dy
		=\sum_{m=0}^{n-2}a_m(f)\gamma_m(x),
	\end{displaymath}
	where $\gamma_m(x)$ is an algebraic polynomial
	of degree not greater than $n-2$,
	and $\gamma_m(x)\equiv0$ for $n=0$ or $n=1$.
\end{lmm}

Lemma~\ref{lm:T*P1} was proved in~\cite{potapov:mat-99}.

\begin{lmm}\label{lm:T*-Q}
	Let~$q$ and~$m$ be given natural numbers.
	Let~$f$ be an integrable function on~$[-1,1]$
	with a weight~$\Si{x}^2$.
	Then for every natural numbers~$l$ and~$r$ $(l\le r)$
	the function
	\begin{displaymath}
		Q_1^{(l)}(x)=\int_0^\pi\dots\int_0^\pi\T{1;\arr{\cos t}l}l{f,x}
			\prod_{s=1}^r\krn{t_s}\sin^3t_s\,dt_1\dots dt_r
	\end{displaymath}
	is an algebraic polynomial
	of degree not greater than $(q+2)\*(m-1)$.
\end{lmm}

\begin{proof}
	In this paper we denote, for simplicity,
	\begin{displaymath}
		A(t):=\krn t.
	\end{displaymath}
	Since
	\begin{displaymath}
		A(t_s)=\sum_{k=0}^{(q+2)(m-1)}a_k\cos kt_s
		=\sum_{k=0}^{(q+2)(m-1)}b_k(\cos t_s)^k,
	\end{displaymath}
	it follows that
	\begin{multline*}
		A(t_s)\sin^2t_s
		  =\sum_{k=0}^{(q+2)(m-1)}b_k(\cos t_s)^k\prn{1-\cos^2t_s}
			=\sum_{k=0}^{(q+2)(m-1)+2}c_k(\cos t_s)^k\\
		=\sum_{k=0}^{(q+2)(m-1)+2}\alpha_k P_k^{(1,1)}(\cos t_s)
		  \quad(s=1,2,\dots,r).
	\end{multline*}
	Hence we have
	\begin{multline*}
		Q_1^{(l)}(x)
		  =\sum_{k=0}^{(q+2)(m-1)+2}\alpha_k
			\int_0^\pi\dots\int_0^\pi
				\prod_{\substack{s=1\\ s\ne l}}^r A(t_s)
					\sin^3t_s\,dt_1\dots dt_{l-1}\,dt_{l+1}\dots dt_r\\
		\times\int_0^\pi\T{1;\arr{\cos t}l}l{f,x}
			  P_k^{(1,1)}(\cos t_l)\sin t_l\,dt_l.
	\end{multline*}
	Let
	\begin{multline*}
		\varphi_{l,k}(x):=\int_0^\pi\T{1;\arr{\cos t}l}l{f,x}
			  P_k^{(1,1)}(\cos t_l)\sin t_l\,dt_l\\
		=\int_0^\pi\T{1;\cos t_l}{}{\T{1;\arr{\cos t}{l-1}}{l-1}{f,x},x}
			  P_k^{(1,1)}(\cos t_l)\sin t_l\,dt_l.
	\end{multline*}
	Substituting $y=\cos t_l$ we obtain
	\begin{displaymath}
		\varphi_{l,k}(x)
		=\int_{-1}^1\T{1;y}{}{\T{1;\arr{\cos t}{l-1}}{l-1}{f,x},x}
			P_k^{(1,1)}(y)\,dy.
	\end{displaymath}
	Then, by Lemma~\ref{lm:T*P1},
	\begin{displaymath}
		\varphi_{l,k}(x)
		=\sum_{m=0}^{k-2}\gamma_m(x)
		  \int_{-1}^1\T{1;\arr{\cos t}{l-1}}{l-1}{f,R}
			  \Px m(R)\Si{R}^2\,dR.
	\end{displaymath}
	On the bases of Corollary~\ref{cr:T*fg},
	we conclude that
	$\T{1;\arr{\cos t}{l-1}}{l-1}{f,R}\in\Lmu$.
	Applying now $l-1$ times Lemma~\ref{lm:T*P}
	we obtain
	\begin{multline*}
		\varphi_{l,k}(x)\\
		=\sum_{m=0}^{k-2}\gamma_m(x)
		  \int_{-1}^1\T{1;\arr{\cos t}{l-2}}{l-2}{f,R}
			  \T{1;\cos t_{l-1}}{}{\Px m,R}\Si{R}^2\,dR\\
		=\sum_{m=0}^{k-2}\gamma_m(x)\Py m(\cos t_{l-1})
			\int_{-1}^1\T{1;\arr{\cos t}{l-2}}{l-2}{f,R}
			  \Px m(R)\Si{R}^2\,dR\\
		=\sum_{m=0}^{k-2}\gamma_m(x)
		  \Py m(\cos t_1)\dots\Py m(\cos t_{l-1})
			\int_{-1}^1f(R)\Px m(R)\Si{R}^2\,dR\\
		=\sum_{m=0}^{k-2}\gamma_m(x)a_m(f)
		  \prod_{s=1}^{l-1}\Py m(\cos t_s),
	\end{multline*}
	where $a_m(f)$ is the Fourier--Jacobi coefficient
	of the function~$f$ with respect to the system
	$\brc{\Px m(x)}_{m=0}^\infty$.
	Substituting the last expression of $\varphi_{l,k}(x)$
	in the formula above for $Q_1^{(l)}(x)$
	we get
	\begin{displaymath}
		Q_1^{(l)}(x)
		=\sum_{k=0}^{(q+2)(m-1)+2}\alpha_k
		  \sum_{m=0}^{k-2}\beta_m\gamma_m(x).
	\end{displaymath}
	Since $\gamma_m(x)$ is from~$\mathcal P_{k-1}$
	for $k\ge2$ and $\gamma_m(x)\equiv0$ for $k=0$ and $k=1$,
	then the last equality yields that
	$Q_1^{(l)}(x)$ is an algebraic polynomial
	of degree not greater than $(q+2)\*(m-1)$.
	
	Lemma~\ref{lm:T*-Q} is proved.
\end{proof}

\begin{lmm}\label{lm:S-Q}
	Let~$q$ and~$m$ be given natural numbers.
	Let~$f$ be an integrable function on $[-1,1]$
	with a weight $\Si{x}^2$.
	For every natural numbers~$l$ and~$r$ $(l\le r)$
	the function
	\begin{displaymath}
		Q_2^{(l)}(x)
		=\int_0^\pi\dots\int_0^\pi\T{2;\arr{\cos t}l}l{f,x}
			\prod_{s=1}^r A(t_s)\sin^5t_s\,dt_1\dots dt_r
	\end{displaymath}
	is an algebraic polynomial of degree not greater than
	$(q+2)\*(m-1)$.
\end{lmm}

\begin{proof}
	As shown in Lemma~\ref{lm:T*-Q},
	\begin{displaymath}
		A(t_s)=\sum_{k=0}^{(q+2)(m-1)}b_k(\cos t_s)^k
		=\sum_{k=0}^{(q+2)(m-1)}\beta_k\Px k(\cos t_s)
		\quad(s=1,2,\dots,r).
	\end{displaymath}
	Hence
	\begin{multline*}
		Q_2^{(l)}(x)
		=\sum_{k=0}^{(q+2)(m-1)}\beta_k\int_0^\pi\dots\int_0^\pi
			\prod_{\substack{s=1\\ s\ne l}}^r A(t_s)
			  \sin^5t_s\,dt_1\dots dt_{l-1}\,dt_{l+1}\dots dt_r\\
		\times\int_0^\pi\T{2;\arr{\cos t}l}l{f,x}
		  \Px k(\cos t_l)\sin^5 t_l\,dt_l.
	\end{multline*}
	Let
	\begin{multline*}
		\psi_{l,k}(x)
		  :=\int_0^\pi\T{2;\arr{\cos t}l}l{f,x}
			\Px k(\cos t_l)\sin^5 t_l\,dt_l\\
		=\int_0^\pi\T{2;\cos t_l}{}{\T{2;\arr{\cos t}{l-1}}{l-1}{f,x},x}
			  \Px k(\cos t_l)\sin^5 t_l\,dt_l.
	\end{multline*}
	Substituting $y=\cos t_l$ we obtain
	\begin{displaymath}
		\psi_{l,k}(x)
		=\int_{-1}^1\T{2;y}{}{\T{2;\arr{\cos t}{l-1}}{l-1}{f,x},x}
			\Px k(y)\Si{y}^2\,dy.
	\end{displaymath}
	Since the operator $\T{2;y}{}{f,x}$ is symmetric
	with respect to~$x$ and~$y$
	(i.e.\ $\T{2;y}{}{g,x}=\T{2;x}{}{g,y}$ for every function~$g$),
	we have
	\begin{displaymath}
		\psi_{l,k}(x)
		=\int_{-1}^1\T{2;x}{}{\T{2;\arr{\cos t}{l-1}}{l-1}{f,y},y}
			\Px k(y)\Si{y}^2\,dy.
	\end{displaymath}
	Note that, in view of Corollary~\ref{cr:T*fg},
	$\T{2;\arr{\cos t}{l-1}}{l-1}{f,y}\in\Lmu$.
	Then, by Lemma~\ref{lm:T*fg},
	\begin{displaymath}
		\psi_{l,k}(x)
		=\int_{-1}^1\T{2;\arr{\cos t}{l-1}}{l-1}{f,y}
			\T{2;x}{}{\Px k,y}\Si{y}^2\,dy.
	\end{displaymath}
	Using the property of the operator~$T_{2;x}$
	described in Lemma~\ref{lm:T*P}
	we get
	\begin{displaymath}
		\psi_{l,k}(x)
		=\Px k(x)\int_{-1}^1\T{2;\arr{\cos t}{l-1}}{l-1}{f,y}
			\Px k(y)\Si{y}^2\,dy.
	\end{displaymath}
	Now we apply $l-1$ times Lemma~\ref{lm:T*P}
	and arrive at the expression
	\begin{multline*}
		\psi_{l,k}(x)
		  =\Px k(x)\Px k(\cos t_1)\dots\Px k(\cos t_{l-1})\\
		\times\int_{-1}^1f(y)\Px k(y)\Si{y}^2\,dy
		  =\Px k(x)a_k(f)\prod_{s=1}^{l-1}\Px k(\cos t_s).
	\end{multline*}
	where $a_k(f)$ is the Fourier--Jacobi coefficient
	of the function~$f$ with respect to the system
	$\brc{\Px k(x)}_{k=0}^\infty$.
	Substituting the last expression of $\psi_{l,k}(x)$
	into the formula for~$Q_2^{(l)}(x)$ above,
	we finally get
	\begin{displaymath}
		Q_2^{(l)}(x)=\sum_{k=0}^{(q+2)(m-1)}\delta_k\Px k(x).
	\end{displaymath}
	Since~$\Px k(x)$ belongs to~$\mathcal P_{k+1}$,
	it is seen from the last identity
	that~$Q_2^{(l)}(x)$ is an algebraic polynomial
	of degree not greater than $(q+2)\*(m-1)$.
	
	The lemma is proved.
\end{proof}

\begin{lmm}\label{lm:properties-T}
	The operator~$T_y$ has the following properties
	\begin{enumerate}
		\item $\T y{}{f,x}$ is linear on~$f$;
		\item $\T1{}{f,x}=f(x)$;
		\item $\T y{}{\Px n,x}=\Px n(x)R_n(y) \quad(n=0,1,\dotsc)$,\\
		  where $R_n(y)=\Py n(y)+\frac32\Si y\Px n(y)$;
		\item $\T y{}{1,x}=1$;
		\item $a_k\prn{\T y{}{f,x}}=R_k(y)a_k(f) \quad(k=0,1,\dotsc)$.
	\end{enumerate}
\end{lmm}

Lemma~\ref{lm:properties-T}
was proved in~\cite{potapov:mat-99}.

\begin{cor}\label{cr:properties-T}
	If $P_n(x)$ is an algebraic polynomial
	of degree not greater than $n-1$,
	then for every natural number~$r$
	and any fixed $y_1,y_2,\allowbreak\dots,y_r$,
	the functions $\T{\arr y r}r{P_n,x}$
	are algebraic polynomials of~$x$
	of degree not greater than $n-1$.
\end{cor}

\begin{lmm}\label{lm:elementary}
	If $-1\le x\le1$, $-1\le z\le1$, $0\le t\le\pi$
	and $R=x\*\cos t-z\*\sqrt{1-x^2}\*\sin t$,
	then $-1\le R\le1$
	and
	\begin{gather*}
		\Si x\Si z\le\Si R,\\
		\Si y\Si z\le\Si R,\\
		\prn{x\sqrt{1-y^2}+yz\sqrt{1-x^2}}^2\le\Si R,\\
		1-x^2\le C\prn{1-R^2+t^2},\\
		1-x\le C\prn{1-R+t^2},\\
		1+x\le C\prn{1+R+t^2},
	\end{gather*}
	where $y=\cos t$
	and~$C$ is an absolute constant.
\end{lmm}

Lemma~\ref{lm:elementary}
was proved in~\cite{potapov:mat-99}
and~\cite{potapov:trudy-81}.

\begin{lmm}\label{lm:inequality}
	Let~$p$, $\alpha$, $\beta$ and~$\gamma$
	be given numbers such that $\allp$,
	$\gamma=\min\{\alpha,\beta\}$,
	and
	\begin{alignat*}2
		\gamma &>1-\frac1{2p} &\quad &\text{for $1\le p<\infty$},\\
		\gamma &\ge1          &\quad &\text{for $p=\infty$}.
	\end{alignat*}
	Let~$\varepsilon$, $0<\varepsilon<\frac12$,
	be an arbitrary number.
	Define
	\begin{displaymath}
		\gamma_1=
		\begin{cases}
		\alpha-\beta, & \text{if $\alpha>\beta$}\\
		0,            & \text{if $\alpha\le\beta$},
		\end{cases}
		\quad
		\gamma_2=
		\begin{cases}
		0,            & \text{if $\alpha>\beta$}\\
		\beta-\alpha, & \text{if $\alpha\le\beta$},
		\end{cases}
	\end{displaymath}
	and for $1<p\le\infty$ let
	\begin{displaymath}
		\gamma_3=
		\begin{cases}
		\gamma-\frac32+\frac1{2p}+\varepsilon,
		  & \text{if $\gamma\ge\frac32-\frac1{2p}$}\\
		0,
		  & \text{if $\gamma<\frac32-\frac1{2p}$},
		\end{cases}
	\end{displaymath}
	while, for $p=1$,
	\begin{displaymath}
		\gamma_3=
		\begin{cases}
		\gamma-1, & \text{if $\gamma\ge1$}\\
		0,        & \text{if $\gamma<1$}.
		\end{cases}
	\end{displaymath}
	Let $R=x\cos t-z\sqrt{1-x^2}\sin t$.
	Then, for every measurable function~$f$ on $[-1,1]$
	the following inequality holds
	\begin{multline*}
		\norm{\frac1{1-x^2}
		  \int_{-1}^1\Si{R}|f(R)|\frac{dz}{\sqrt{1-z^2}}}\\
		\le C
		  \Big(
			\norm{f}
			+t^{2(\gamma_1+\gamma_2)}
			  \|f\|_{p,\alpha-\gamma_1,\beta-\gamma_2}
			  +t^{2\gamma_3}\|f\|_{p,\alpha-\gamma_3,\beta-\gamma_3}\\
			  +t^{2(\gamma_1+\gamma_2+\gamma_3)}
				\|f\|_{p,\alpha-\gamma_1-\gamma_3,\beta-\gamma_2-\gamma_3}
		  \Big),
	\end{multline*}
	where the constant~$C$ does not depend on~$f$ and~$t$.
\end{lmm}

\begin{proof}
	If at least one of the terms on the right-hand side
	of the inequality is not finite,
	then the lemma is obvious.
	
	Suppose now that all the terms on the right-hand side
	of the inequality are finite.
	
	Let $\alpha\ge\beta$.
	We first consider the case $1\le p<\infty$.
	Clearly
	\begin{multline}\label{eq:I}
		I:=\norm{
			\frac1{1-x^2}\int_{-1}^1\Si{R}|f(R)|\frac{dz}{\sqrt{1-z^2}}
		  }^p\\
		=\int_{-1}^1
		  \left|\int_{-1}^1|f(R)|\Si{z}^{-1/2}\Si R\,dz\right|^p
			(1-x)^{p(\alpha-1)}(1+x)^{p(\beta-1)}\,dx.
	\end{multline}
	
	If $p=1$, then
	\begin{displaymath}
		I\le\int_{-1}^1\int_{-1}^1|f(R)|\delta\,dz\,dx,
	\end{displaymath}
	where
	\begin{displaymath}
		\delta=\Si{z}^{-1/2}\Si R\Si{x}^{\beta-1}(1-x)^{\alpha-\beta}.
	\end{displaymath}
	
	Let $\beta<1$.
	Then, in view of Lemma~\ref{lm:elementary},
	\begin{multline*}
		\delta=\Si{z}^{\frac12-\beta}
			\prn{\Si z\Si x}^{\beta-1}\Si R(1-x)^{\alpha-\beta}\\
		\le\Cn\Si{z}^{-1/2}\prn{\Si z\Si x}^{\beta-1}
			\Si R\prn{1-R+t^2}^{\alpha-\beta}\\
		=\lastC\delta_1(x,z,R).
	\end{multline*}
	
	Suppose that $\beta\ge1$.
	Making use of Lemma~\ref{lm:elementary} we see that
	\begin{multline*}
		\delta\le\Cn\Si{z}^{-1/2}\Si R
			\prn{1-R^2+t^2}^{\beta-1}\prn{1-R+t^2}^{\alpha-\beta}\\
		=\lastC\delta_2(x,z,R).
	\end{multline*}
	
	Incorporating these estimates for~$\delta$
	we get the inequality
	\begin{displaymath}
		I\le\Cn\int_{-1}^1\int_{-1}^1|f(R)|\delta_k(x,z,R)\,dz\,dx
		\quad(k=1,2),
	\end{displaymath}
	which, after the change of variables~\eqref{eq:x-z},
	takes the form
	\begin{displaymath}
		I\le\lastC\int_{-1}^1\int_{-1}^1|f(R)|\delta_k(R,V,R)\,dV\,dR
		\quad(k=1,2).
	\end{displaymath}
	
	Set
	\begin{multline*}
		\hat\delta_1(R)
		  :=\Si{R}^\beta\prn{1-R+t^2}^{\alpha-\beta}\\
		\le\Cn\prn{(1-R)^\alpha(1+R)^\beta
		  +t^{2(\alpha-\beta)}\Si{R}^\beta},
	\end{multline*}
	\begin{multline*}
		\hat\delta_2(R)
		  :=\Si R\prn{1-R^2+t^2}^{\beta-1}
			\prn{1-R+t^2}^{\alpha-\beta}\\
		\le\Cn\Big(
			  (1-R)^\alpha(1+R)^\beta
			  +t^{2(\alpha-\beta)}\Si{R}^\beta\\
		+t^{2(\beta-1)}(1-R)^{\alpha-\beta+1}(1+R)
			  +t^{2(\alpha-1)}\Si R
			\Big).
	\end{multline*}
	Then clearly
	\begin{displaymath}
		I\le\Cn\int_{-1}^1|f(R)|\hat\delta_k(R)\,dR
		\quad(k=1,2).
	\end{displaymath}
	
	The last inequality
	and the estimates for $\hat\delta_k(R)$,
	given above, yield
	\begin{multline*}
		I\le\Cn\big(
			\|f\|_{1,\alpha,\beta}
			+t^{2\gamma_1}\|f\|_{1,\alpha-\gamma_1,\beta}
			+t^{2\gamma_3}\|f\|_{1,\alpha-\gamma_3,\beta-\gamma_3}\\
		+t^{2(\gamma_1+\gamma_3)}
			  \|f\|_{1,\alpha-\gamma_1-\gamma_3,\beta-\gamma_3}
		  \big),
	\end{multline*}
	where the constant~$\lastC$ does not depend on~$f$ and~$t$.
	Hence the lemma is true in the case $p=1$.
	
	Assume now that $1<p<\infty$.
	Applying H\"older's inequality to the inside integral
	in~\eqref{eq:I}
	we get
	\begin{multline*}
		I=\int_{-1}^1
			\left|
			  \int_{-1}^1|f(R)|\Si{z}^{-\frac12-\frac1p+1-b}\Si R
					\Si{z}^{-\prn{1-\frac1p}+b}\,dz
				\right|^p\\
		\times(1-x)^{p(\alpha-1)}(1+x)^{p(\beta-1)}\,dx
		  \le\Cn\int_{-1}^1\int_{-1}^1|f(R)|^p\varkappa\,dz\,dx,
	\end{multline*}
	where
	\begin{displaymath}
		\varkappa
		=\Si{z}^{-1+p\prn{\frac12-b}}\Si{R}^p
		  \Si{x}^{p(\beta-1)}(1-x)^{p(\alpha-\beta)},
	\end{displaymath}
	$b$~is an arbitrary positive number,
	the constant~$\lastC$ does not depend on~$t$
	and the function~$f$.
	
	Let $\beta<\frac32-\frac1{2p}$.
	Put $b=\frac32-\frac1{2p}-\beta$.
	Applying Lemma~\ref{lm:elementary}
	we derive the estimate
	\begin{multline*}
		\varkappa
		  \le\Cn\Si{z}^{-1/2}\prn{\Si z\Si x}^{p(\beta-1)}
			\Si{R}^p\prn{1-R+t^2}^{p(\alpha-\beta)}\\
		=\lastC\varkappa_1(x,z,R).
	\end{multline*}
	
	Let $\beta\ge\frac32-\frac1{2p}$.
	Put $b=\varepsilon$,
	where~$\varepsilon$ is an arbitrary number
	belonging to the interval $0<\varepsilon<\frac12$.
	Again by Lemma~\ref{lm:elementary} we see that
	\begin{multline*}
		\varkappa
		  \le\Cn\Si{z}^{-1/2}
			\prn{\Si z\Si x}^{-\frac12+p\prn{\frac12-\varepsilon}}
			\Si{R}^p\\
		\times\prn{1-R^2+t^2}^{p\prn{\beta-\frac32+\frac1{2p}+\varepsilon}}
			\prn{1-R+t^2}^{p(\alpha-\beta)}
		  =\lastC\varkappa_2(x,z,R).
	\end{multline*}
	
	Using these estimates for~$\varkappa$
	we get the inequality
	\begin{displaymath}
		I\le\Cn\int_{-1}^1\int_{-1}^1|f(R)|^p
		  \varkappa_k(x,z,R)\,dz\,dx
		\quad(k=1,2).
	\end{displaymath}
	and consequently
	(after the changes of variables~\eqref{eq:x-z}),
	\begin{displaymath}
		I\le\lastC\int_{-1}^1\int_{-1}^1|f(R)|^p
		  \varkappa_k(R,V,R)\,dV\,dR
		\quad(k=1,2).
	\end{displaymath}
	
	Set
	\begin{multline*}
		\hat\varkappa_1(R)
		  :=\Si{R}^{p\beta}\prn{1-R+t^2}^{p(\alpha-\beta)}\\
		\le\Cn\prn{(1-R)^{p\alpha}(1+R)^{p\beta}
			+t^{2p(\alpha-\beta)}\Si{R}^{p\beta}},
	\end{multline*}
	\begin{multline*}
		\hat\varkappa_2(R)
		  :=\Si{R}^{-\frac12+p\prn{\frac32-\varepsilon}}
			  \prn{1-R^2+t^2}^{p\prn{\beta-\frac32+\frac1{2p}+\varepsilon}}
			\prn{1-R+t^2}^{p(\alpha-\beta)}\\
		\le\Cn
			\bigg(
			  (1-R)^{p\alpha}(1+R)^{p\beta}
			  +t^{2p(\alpha-\beta)}\Si{R}^{p\beta}\\
		+t^{2p\prn{\beta-\frac32+\frac1{2p}+\varepsilon}}
				(1-R)^{p\prn{\alpha-\beta+\frac32-\frac1{2p}-\varepsilon}}
				(1+R)^{p\prn{\frac32-\frac1{2p}-\varepsilon}}\\
		+t^{2p\prn{\alpha-\frac32+\frac1{2p}+\varepsilon}}
				\Si{R}^{p\prn{\frac32-\frac1{2p}-\varepsilon}}
			\bigg).
	\end{multline*}
	Then clearly
	\begin{displaymath}
		I\le\Cn\int_{-1}^1|f(R)|^p\hat\varkappa_k(R)\,dR
		\quad(k=1,2).
	\end{displaymath}
	From the last inequality
	and the estimates of $\hat\varkappa_k(R)$
	we obtain
	\begin{multline*}
		I\le\Cn
		  \Big(
			\norm{f}^p
			+t^{2p\gamma_1}\|f\|^p_{p,\alpha-\gamma_1,\beta}
			  +t^{2p\gamma_3}\|f\|^p_{p,\alpha-\gamma_3,\beta-\gamma_3}\\
		+t^{2p(\gamma_1+\gamma_3)}
			  \|f\|^p_{p,\alpha-\gamma_1-\gamma_3,\beta-\gamma_3}
		  \Big),
	\end{multline*}
	where the constant~$\lastC$ does not depend on~$f$ and~$t$.
	This shows that the lemma is true
	in the case $1<p<\infty$ as well.
	
	Now let $p=\infty$.
	Consider the integral
	\begin{multline*}
		J:=\int_{-1}^1|f(R)|\Si{z}^{-1/2}\Si R
		  (1-x)^{\alpha-1}(1+x)^{\beta-1}\,dz\\
		=\int_{-1}^1|f(R)|\lambda\,dz,
	\end{multline*}
	where
	\begin{displaymath}
		\lambda
		=\Si{z}^{-1+b}\Si R\Si{x}^{\beta-1}
		  (1-x)^{\alpha-\beta}\Si{z}^{\frac12-b}
	\end{displaymath}
	and $b$~is an arbitrary positive number.
	
	Let $\beta<\frac32$.
	Put $b=\frac32-\beta$.
	Applying the estimate from Lemma~\ref{lm:elementary}
	we get
	\begin{multline*}
		\lambda
		  =\Si{z}^{\frac12-\beta}\Si R\prn{\Si z\Si x}^{\beta-1}
			(1-x)^{\alpha-\beta}\\
		\le\Cn\Si{z}^{\frac12-\beta}\Si{R}^\beta
			\prn{1-R+t^2}^{\alpha-\beta}
		  =\lastC\Si{z}^{\frac12-\beta}\lambda_1(R).
	\end{multline*}
	
	Let $\beta\ge\frac32$.
	Put $b=\varepsilon$,
	where~$\varepsilon$ is an arbitrary number
	from the interval $0<\varepsilon<\frac12$.
	Applying again Lemma~\ref{lm:elementary}
	we see that
	\begin{multline*}
		\lambda
		  =\Si{z}^{-1+\varepsilon}\prn{\Si z\Si x}^{\frac12-\varepsilon}
			\Si R\Si{x}^{\beta-\frac32+\varepsilon}(1-x)^{\alpha-\beta}\\
		\le\Cn\Si{z}^{-1+\varepsilon}\Si{R}^{\frac32-\varepsilon}
			\prn{1-R^2+t^2}^{\beta-\frac32+\varepsilon}
			\prn{1-R+t^2}^{\alpha-\beta}\\
		=\lastC\Si{z}^{-1+\varepsilon}\lambda_2(R).
	\end{multline*}
	
	Using these estimates for~$\lambda$
	and taking into account the relations
	\begin{displaymath}
		\lambda_1(R)
		\le\Cn\prn{(1-R)^\alpha(1+R)^\beta
		  +t^{2(\alpha-\beta)}\Si{R}^\beta},
	\end{displaymath}
	\begin{multline*}
		\lambda_2(R)
		  \le\Cn
			\Big(
			  (1-R)^\alpha(1+R)^\beta
				+t^{2(\alpha-\beta)}\Si{R}^\beta\\
		+t^{2\prn{\beta-\frac32+\varepsilon}}
				(1-R)^{\alpha-\beta+\frac32-\varepsilon}
				(1+R)^{\frac32-\varepsilon}
			  +t^{2\prn{\alpha-\frac32+\varepsilon}}
				\Si{R}^{\frac32-\varepsilon}
			\Big),
	\end{multline*}
	for $k=1,2$
	we obtain
	\begin{multline*}
		J\le\Cn\max_{-1\le R\le1}|f(R)|\lambda_k(R)
		  \le\Cn
			\big(
			  \|f\|_{\infty,\alpha,\beta}
			  +t^{2\gamma_1}\|f\|_{\infty,\alpha-\gamma_1,\beta}\\
		+t^{2\gamma_3}\|f\|_{\infty,\alpha-\gamma_3,\beta-\gamma_3}
				+t^{2(\gamma_1+\gamma_3)}
				  \|f\|_{\infty,\alpha-\gamma_1-\gamma_3,\beta-\gamma_3}
			  \big),
	\end{multline*}
	where the constant~$\lastC$ does not depend on~$f$ and~$t$.
	This proves the lemma for $p=\infty$.
	
	Thus, the lemma is proved for $\alpha\ge\beta$.
	The case $\alpha\le\beta$ goes similarly.
	We omit the details.
	The proof is complete.
\end{proof}

\subsection{}

\begin{thm}\label{th:bound-T}
	Let~$p$, $\alpha$, $\beta$ and~$\gamma$ be given numbers
	such that $\allp$, $\gamma=\min\{\alpha,\beta\}$.
	Assume that
	\begin{alignat*}2
		\gamma &>1-\frac1{2p} &\quad &\text{for $1\le p<\infty$},\\
		\gamma &\ge1          &\quad &\text{for $p=\infty$}.
	\end{alignat*}
	Let~$\varepsilon$ be an arbitrary number
	belonging to the interval $0<\varepsilon<\frac12$.
	Let
	\begin{displaymath}
		\gamma_1=
		\begin{cases}
		\alpha-\beta, & \text{if $\alpha>\beta$}\\
		0,            & \text{if $\alpha\le\beta$},
		\end{cases}
		\quad
		\gamma_2=
		\begin{cases}
		0,            & \text{if $\alpha>\beta$}\\
		\beta-\alpha, & \text{if $\alpha\le\beta$}.
		\end{cases}
	\end{displaymath}
	Set for $1<p\le\infty$
	\begin{displaymath}
		\gamma_3=
		\begin{cases}
		\gamma-\frac32+\frac1{2p}+\varepsilon,
		  & \text{for $\gamma\ge\frac32-\frac1{2p}$}\\
		0,
		  & \text{for $\gamma<\frac32-\frac1{2p}$},
		\end{cases}
	\end{displaymath}
	and
	\begin{displaymath}
		\gamma_3=
		\begin{cases}
		\gamma-1, & \text{for $\gamma\ge1$}\\
		0,        & \text{for $\gamma<1$}
		\end{cases}
	\end{displaymath}
	for $p=1$.
	Then the following inequality holds true
	\begin{multline*}
		\norm{\hatT t{}{f,x}}
		  \le C\Big(
			  \norm{f}
				+t^{2(\gamma_1+\gamma_2)}
				  \normpar f{p,\alpha-\gamma_1,\beta-\gamma_2}\\
		+t^{2\gamma_3}\normpar f{p,\alpha-\gamma_3,\beta-\gamma_3}
				+t^{2(\gamma_1+\gamma_2+\gamma_3)}
				  \normpar f{p,\alpha-\gamma_1-\gamma_3,
					\beta-\gamma_2-\gamma_3}
			  \Big),
	\end{multline*}
	where the constant~$C$ does not depend on~$f$ and~$t$.
\end{thm}

\begin{proof}
	We have
	\begin{displaymath}
		\norm{\hatT t{}{f,x}}
		\le\frac1\pi\norm{\frac1{1-x^2}
			\int_{-1}^1A|f(R)|\frac{dz}{\sqrt{1-z^2}}},
	\end{displaymath}
	where $R=x\cos t-z\sqrt{1-x^2}\sin t$,
	\begin{displaymath}
		A=1-R^2-2\Si x\sin^2t+4\Si x\Si{z}^2\sin^2t.
	\end{displaymath}
	Using Lemma~\ref{lm:elementary}
	we get
	\begin{displaymath}
		A\le1-R^2+2\Si R+4\Si{R}^2\le7\Si R.
	\end{displaymath}
	Hence
	\begin{displaymath}
		\norm{\hatT t{}{f,x}}
		\le\frac7\pi\norm{\frac1{1-x^2}
			\int_{-1}^1\Si{R}|f(R)|\frac{dz}{\sqrt{1-z^2}}}.
	\end{displaymath}
	Now the theorem follows from Lemma~\ref{lm:inequality}.
\end{proof}

\begin{thm}\label{th:T-Q}
	Let~$q$, $m$ and~$r$ be given natural numbers
	and let $f\in\Lmu$.
	The function
	\begin{multline*}
		Q(x)=\frac1{(\gamma_m)^r}
				\int_0^\pi\dots\int_0^\pi
				  \prn{\Dl{\arr t r}r{f,x}-(-1)^r f(x)}\\
		\times\prod_{s=1}^rA(t_s)\sin^3t_s\,dt_1\dots dt_r,
	\end{multline*}
	where
	\begin{displaymath}
		\gamma_m=\int_0^\pi A(t)\sin^3t\,dt,
	\end{displaymath}
	is an algebraic polynomial
	of degree not greater than $(q+2)\*(m-1)$.
\end{thm}

\begin{proof}
	To prove the theorem it is sufficient to show that
	for every $l=1,\allowbreak\dots,r$
	the function
	\begin{displaymath}
		Q^{(l)}(x)
		=\frac1{(\gamma_m)^r}
		  \int_0^\pi\dots\int_0^\pi
			\T{\arr{\cos t}l}l{f,x}
			  \prod_{s=1}^rA(t_s)\sin^3t_s\,dt_1\dots dt_r
	\end{displaymath}
	is an algebraic polynomial
	of degree not greater than $(q+2)\*(m-1)$.
	
	It is obvious that the function $Q^{(l)}(x)$
	can be written in the form
	\begin{displaymath}
		Q^{(l)}(x)
		=\frac1{(\gamma_m)^r}\prn{Q_1^{(l)}(x)+\frac32Q_2^{(l)}(x)},
	\end{displaymath}
	where $Q_1^{(l)}(x)$ and $Q_2^{(l)}(x)$
	are the functions
	from Lemmas~\ref{lm:T*-Q} and~\ref{lm:S-Q}, respectively.
	But then it follows from Lemmas~\ref{lm:T*-Q}
	and~\ref{lm:S-Q}
	that $Q^{(l)}(x)$ is an algebraic polynomial
	of degree not greater than $(q+2)\*(m-1)$.
	
	The theorem is proved.
\end{proof}

\begin{thm}\label{th:HsubE}
	Let~$p$, $\alpha$, $\beta$, $r$ and~$\lambda$
	be given numbers such that $\allp$, $\lambda>0$, $r\in\numN$.
	Assume that
	\begin{alignat*}3
		\alpha &\le2, 		  &\quad \beta &\le2
		  &\quad &\text{for $p=1$},\\
		\alpha &<3-\frac1p, &\quad \beta &<3-\frac1p
		  &\quad &\text{for $1<p\le\infty$}.
	\end{alignat*}
	Let $f\in\Lp$ and
	\begin{displaymath}
		\w\le M\delta^\lambda.
	\end{displaymath}
	Then
	\begin{displaymath}
		\E\le CMn^{-\lambda},
	\end{displaymath}
	where the constant~$C$
	does not depend on~$f$, $M$ and~$n$ $(n\in\numN)$.
\end{thm}

\begin{proof}
	Under the conditions of the theorem,
	if $f\in\Lp$,
	then $f\in\Lmu$.
	Indeed, for $p=1$ we have
	\begin{displaymath}
		\normpar f{1,2,2}
		=\int_{-1}^1|f(x)|(1-x)^\alpha(1+x)^\beta
		  (1-x)^{2-\alpha}(1+x)^{2-\beta}\,dx
		\le\Cn\normpar f{1,\alpha,\beta}
	\end{displaymath}
	provided $\alpha\le2$ and $\beta\le2$.
	For $1<p<\infty$, by H\"older's inequality,
	\begin{multline*}
		\normpar f{1,2,2}
		  \le\brc{
			  \int_{-1}^1|f(x)|^p(1-x)^{p\alpha}(1+x)^{p\beta}\,dx
			}^{1/p}\\
		\times\brc{
			  \int_{-1}^1(1-x)^{(2-\alpha)\frac p{p-1}}
				(1+x)^{(2-\beta)\frac p{p-1}}\,dx
			}^{\frac{p-1}p}
		  =\Cn\norm f
	\end{multline*}
	for $\alpha<3-\frac1p$ and $\beta<3-\frac1p$.
	For $p=\infty$ we have
	\begin{displaymath}
		\normpar f{1,2,2}
		\le\normpar f{\infty,\alpha,\beta}
		  \int_{-1}^1(1-x)^{2-\alpha}(1+x)^{2-\beta}\,dx
		=\Cn\normpar f{\infty,\alpha,\beta}
	\end{displaymath}
	provided $\alpha<3$ and $\beta<3$.
	
	We choose a natural number~$q$ such that $2q>\lambda$,
	and for each $n\in\numN$ we choose a number $m\in\numN$
	satisfying the condition
	\begin{equation}\label{eq:m}
		\frac{n-1}{q+2}<m\le\frac{n-1}{q+2}+1.
	\end{equation}
	For these~$q$ and~$m$ the polynomial $Q(x)$
	defined in Theorem~\ref{th:T-Q}
	is from~$\mathcal P_n$.
	Hence
	\begin{multline*}
		\E\le\norm{f(x)-(-1)^{r+1}Q(x)}\\
		=\norm{\frac1{(\gamma_m)^r}
		  \int_0^\pi\dots\int_0^\pi\Dl{\arr t r}r{f,x}
			  \prod_{s=1}^rA(t_s)\sin^3t_s\,dt_1\dots dt_r}.
	\end{multline*}
	Applying the generalized inequality of Minkowski
	we obtain
	\begin{multline*}
		\E\le\frac1{(\gamma_m)^r}
		  \int_0^\pi\dots\int_0^\pi\norm{\Dl{\arr t r}r{f,x}}
				\prod_{s=1}^rA(t_s)\sin^3t_s\,dt_1\dots dt_r\\
		\le\frac1{(\gamma_m)^r}
		  \int_0^\pi\dots\int_0^\pi
			  \sup_{|u_i|\le\sum_{j=1}^r t_j}\norm{\Dl{\arr u r}r{f,x}}\\
		\times\prod_{s=1}^r A(t_s)\sin^3t_s\,dt_1\dots dt_r\\
		\le\frac1{(\gamma_m)^r}
		  \int_0^\pi\dots\int_0^\pi\hat\omega_r
			  \bigg(f,\sum_{j=1}^r t_j\bigg)_{p,\alpha,\beta}
			  \prod_{s=1}^rA(t_s)\sin^3t_s\,dt_1\dots dt_r.
	\end{multline*}
	Hence, taking into account the assumptions of the theorem,
	we have
	\begin{multline*}
		\E\le\frac M{(\gamma_m)^r}
		  \int_0^\pi\dots\int_0^\pi
			\bigg(\sum_{j=1}^r t_j\bigg)^\lambda
			  \prod_{s=1}^rA(t_s)\sin^3t_s\,dt_1\dots dt_r\\
		\le\Cn M\sum_{j=1}^r\frac1{(\gamma_m)^r}
		  \int_0^\pi\dots\int_0^\pi t_j^\lambda
			  \prod_{s=1}^rA(t_s)\sin^3t_s\,dt_1\dots dt_r.
	\end{multline*}
	Applying the standard evaluation of Jackson's kernel
	and making use of inequality~\eqref{eq:m}
	we obtain
	\begin{displaymath}
		\E\le\Cn Mm^{-\lambda}\le\Cn Mn^{-\lambda}.
	\end{displaymath}
	
	Theorem~\ref{th:HsubE} is proved.
\end{proof}

\begin{thm}\label{th:EsubH}
	Let~$p$, $\alpha$, $\beta$, $r$ and~$\lambda$
	be given numbers
	such that $\allp$, $r\in\numN$.
	Assume that
	\begin{alignat*}3
		\alpha &>1-\frac1{2p}, &\quad \beta &>1-\frac1{2p}
		  &\quad &\text{for $1\le p<\infty$},\\
		\alpha &\ge1,          &\quad \beta &\ge1
		  &\quad &\text{for $p=\infty$};
	\end{alignat*}
	\begin{displaymath}
		\lambda_0
		=2\max
			\brc{|\alpha-\beta|,\alpha-\frac32+\frac1{2p},
			  \beta-\frac32+\frac1{2p}}
		<\lambda<2r.
	\end{displaymath}
	If $f\in\Lp$ and
	\begin{displaymath}
		\E\le\frac M{n^\lambda},
	\end{displaymath}
	then
	\begin{displaymath}
		\w\le CM\delta^\lambda,
	\end{displaymath}
	where the constant~$C$
	does not depend on~$f$, $M$ and~$\delta$.
\end{thm}

\begin{proof}
	Let $P_n(x)$ be the polynomial from~$\mathcal P_n$
	for which
	\begin{displaymath}
		\norm{f-P_n}=\E \quad(n=1,2,\dotsc).
	\end{displaymath}
	We construct the polynomials $Q_k(x)$ by
	\begin{displaymath}
		Q_k(x)=P_{2^k}(x)-P_{2^{k-1}}(x)
		\quad(k=1,2,\dotsc)
	\end{displaymath}
	and $Q_0(x)=P_1(x)$.
	Since for $k\ge1$ we have
	\begin{multline*}
		\norm{Q_k}=\norm{P_2^k-P_{2^{k-1}}}
		  \le\norm{P_{2^k}-f}+\norm{f-P_{2^{k-1}}}\\
		=\Epar{2^k}f+\Epar{2^{k-1}}f,
	\end{multline*}
	then it follows from the assumptions of the theorem that
	\begin{displaymath}
		\norm{Q_k}\le\Cn M2^{-k\lambda}.
	\end{displaymath}
	
	It is obvious that without lost of generality
	we may assume that $t_s\ne0$ $(s=1,\allowbreak\dots,r)$.
	Next we estimate the quantity
	\begin{displaymath}
		I=\norm{\Dl{\arr t r}r{f,x}}
	\end{displaymath}
	for $0<|t_s|<\delta$ $(s=1,\allowbreak\dots,r)$.
	For every natural number~$N$,
	taking into account
	that the linearity of the operator $\hatT{t_1}{}{f,x}$
	implies the linearity of $\hatT{\arr t r}r{f,x}$,
	i.e.\ the linearity of the difference $\Dl{\arr t r}r{f,x}$,
	we have
	\begin{displaymath}
		I\le\norm{\Dl{\arr t r}r{f-P_{2^N},x}}
		  +\norm{\Dl{\arr t r}r{P_{2^N},x}}.
	\end{displaymath}
	Since $P_{2^N}(x)=\sum_{k=0}^N Q_k(x)$,
	we get
	\begin{displaymath}
		I\le\norm{\Dl{\arr t r}r{f-P_{2^N},x}}
		  +\sum_{k=0}^N\norm{\Dl{\arr t r}r{Q_k,x}}
		:=A+\sum_{k=1}^N I_k.
	\end{displaymath}
	
	Let~$N$ be chosen so that
	\begin{equation}\label{eq:delta}
		\frac\pi{2^N}<\delta\le\frac\pi{2^{N-1}}.
	\end{equation}
	We shall show that
	\begin{equation}\label{eq:J}
		A\le\Cn M\delta^\lambda
	\end{equation}
	and
	\begin{equation}\label{eq:Ik}
		I_k\le\Cn M\delta^{2r}2^{k(2r-\lambda)}.
	\end{equation}
	
	Consider first the quantity~$A$.
	Assume that $r=1$.
	An application of Theorem~\ref{th:bound-T}
	to the function $\varphi(x)=f(x)-P_{2^N}(x)$
	gives
	\begin{multline*}
		\norm{\Dl{t_1}{}{f-P_{2^N},x}}
		  =\norm{\hatT{t_1}{}{\varphi,x}-\varphi(x)}\\
		\le\norm{\hatT{t_1}{}{\varphi,x}}+\norm{\varphi(x)}
		  \le\Cn
			\Big(
			  \norm{\varphi}
				+\delta^{2(\gamma_1+\gamma_2)}
				  \normpar{\varphi}{p,\alpha-\gamma_1,\beta-\gamma_2}\\
		+\delta^{2\gamma_3}
				\normpar{\varphi}{p,\alpha-\gamma_3,\beta-\gamma_3}
				+\delta^{2(\gamma_1+\gamma_2+\gamma_3)}
					\normpar{\varphi}{p,\alpha-\gamma_1-\gamma_3,
					  \beta-\gamma_2-\gamma_3}
				\Big)
	\end{multline*}
	for $|t_1|\le\delta$,
	where the numbers~$\gamma_1$, $\gamma_2$ and~$\gamma_3$
	are chosen as in Theorem~\ref{th:bound-T}.
	Hence, by Lemma~\ref{lm:rho-sigma}
	\begin{multline*}
		\norm{\Dl{t_1}{}{f-P_{2^N},x}}
		  \le\Cn M
			\Big(
			  2^{-N\lambda}
				+\delta^{2(\gamma_1+\gamma_2)}
				  2^{-N(\lambda-2\gamma_1-2\gamma_2)}\\
		+\delta^{2\gamma_3}2^{-N(\lambda-2\gamma_3)}
			  +\delta^{2(\gamma_1+\gamma_2+\gamma_3)}
				2^{-N(\lambda-2\gamma_1-2\gamma_2-2\gamma_3)}
			\Big)
	\end{multline*}
	for $\lambda>\lambda_0+\varepsilon$,
	where the constant~$\lastC$
	does not depend on~$f$, $M$ and~$\delta$.
	Here~$\varepsilon$ is either equal to~$0$
	or is an arbitrary number
	belonging to the interval $0<\varepsilon<\frac12$.
	Therefore, this inequality holds for any $\lambda>\lambda_0$.
	Finally, applying inequality~\eqref{eq:delta} we get
	\begin{displaymath}
		\norm{\Dl{t_1}{}{f-P_{2^N},x}}
		\le\Cn M2^{N\lambda}\le\Cn M\delta^\lambda.
	\end{displaymath}
	Thus inequality~\eqref{eq:J} is proved for $r=1$.
	
	Suppose that
	\begin{displaymath}
		\norm{\Dl{\arr t{r-1}}{r-1}{f-P_{2^N},x}}
		\le\Cn M\delta^\lambda.
	\end{displaymath}
	Then inequality~\eqref{eq:delta} yields
	\begin{multline*}
		\norm{\Dl{\arr t{r-1}}{r-1}{f-P_{2^N},x}}\\
		=\norm{\Dl{\arr t{r-1}}{r-1}{f,x}
			  -\Dl{\arr t{r-1}}{r-1}{P_{2^N},x}}
			\le\Cn\frac M{2^{N\lambda}}.
	\end{multline*}
	Reasoning as above,
	i.e.\ applying first Theorem~\ref{th:bound-T}
	to the function
	\begin{displaymath}
		\Dl{\arr t{r-1}}{r-1}{f-P_{2^N},x},
	\end{displaymath}
	taking into account that by Corollary~\ref{cr:properties-T}
	$\Dl{\arr t{r-1}}{r-1}{P_{2^N},x}$ is an algebraic polynomial
	of degree not greater than $2^N-1$,
	applying Lemma~\ref{lm:rho-sigma},
	and finally inequality~\eqref{eq:delta},
	we obtain that
	\begin{displaymath}
		A=\norm{\Dl{\arr t r}r{f-P_{2^N},x}}
		\le\Cn\delta^\lambda.
	\end{displaymath}
	Inequality~\eqref{eq:J} is proved.
	
	Now we prove inequality~\eqref{eq:Ik}.
	Let
	\begin{displaymath}
		\psi_k(x)=\Dl{\arr t r}r{Q_k,x}.
	\end{displaymath}
	It can be shown that
	\begin{multline*}
		\psi_k(x)
		  =\frac1{2\pi\Si x}\int_0^{t_r}\int_{-u}^u\int_0^\pi
			\bigg(
				A(v)(R'_v)^2
				\frac{d^2}{dR_v^2}\Dl{\arr t{r-1}}{r-1}{Q_k,R_v}\\
		-\prn{A(v)R_v-2A'(v)R'_v}
			  \frac d{dR_v}\Dl{\arr t{r-1}}{r-1}{Q_k,R_v}\\
		+A''(v)\Dl{\arr t{r-1}}{r-1}{Q_k,R_v}
			\bigg)\,d\varphi\,dv\,du,
	\end{multline*}
	where $R_v=x\cos v-\sqrt{1-x^2}\cos\varphi\sin v$,
	\begin{displaymath}
		A(v)=1-R_v^2-2\sin^2v\sin^2\varphi+4\Si x\sin^2v\sin^4\varphi.
	\end{displaymath}
	
	Applying the estimates from Lemma~\ref{lm:elementary}
	and performing the change of variables $z=\cos\varphi$
	we obtain
	\begin{displaymath}
		|\psi_k(x)|
		\le\frac{\Cn}{1-x^2}\int_0^{t_r}\int_{-u}^u\int_{-1}^1
		  B(R_v)\frac{dz}{\sqrt{1-z^2}}\,dv\,du,
	\end{displaymath}
	where
	\begin{multline*}
		B(R_v)
		  =\Si{R_v}^2
			  \left|
				\frac{d^2}{dR_v^2}\Dl{\arr t{r-1}}{r-1}{Q_k,R_v}
			  \right|\\
		+\Si{R_v}
			  \left|\frac d{dR_v}\Dl{\arr t{r-1}}{r-1}{Q_k,R_v}\right|
			  +\left|\Dl{\arr t{r-1}}{r-1}{Q_k,R_v}\right|\\
		=B_1(R_v)+B_2(R_v)+B_3(R_v).
	\end{multline*}
	Therefore, using the generalized Minkowski inequality,
	we get
	\begin{multline}\label{eq:sumBmu}
		I_k=\norm{\psi_k(x)}
		  \le\lastC\int_0^{t_r}\int_{-u}^u\norm{\frac1{1-x^2}
				\int_{-1}^1B(R_v)\frac{dz}{\sqrt{1-z^2}}}\,dv\,du\\
		\le\Cn t_r^2\sup_{|v|\le t_r}
			  \norm{\frac1{1-x^2}
				\int_{-1}^1B(R_v)\frac{dz}{\sqrt{1-z^2}}}\\
		\le\lastC t_r^2\sum_{\mu=1}^3\sup_{|v|\le t_r}
				\norm{\frac1{1-x^2}
				  \int_{-1}^1B_\mu(R_v)\frac{dz}{\sqrt{1-z^2}}}.
	\end{multline}
	Next,
	applying first Lemma~\ref{lm:inequality}
	to the function $B_1(R_v)$,
	then Lemma~\ref{lm:bernshtein-markov},
	and finally inequality~\eqref{eq:delta},
	we get for $|v|\le|t_r|\le\delta$ and $k\le N$
	\begin{multline*}
		\norm{\frac1{1-x^2}
			  \int_{-1}^1B_1(R_v)\frac{dz}{\sqrt{1-z^2}}}\\
		\le\Cn
			\bigg(
			  \norm{\Si x\frac{d^2}{dx^2}\Dl{\arr t{r-1}}{r-1}{Q_k,x}}\\
		+|v|^{2(\gamma_1+\gamma_2)}
				\normpar{\Si x
				  \frac{d^2}{dx^2}\Dl{\arr t{r-1}}{r-1}{Q_k,x}}
					  {p,\alpha-\gamma_1,\beta-\gamma_2}\\
		+|v|^{2\gamma_3}
				\normpar{\Si x
				  \frac{d^2}{dx^2}\Dl{\arr t{r-1}}{r-1}{Q_k,x}}
					  {p,\alpha-\gamma_3,\beta-\gamma_3}\\
		+|v|^{2(\gamma_1+\gamma_2+\gamma_3)}
				\normpar{\Si x
				  \frac{d^2}{dx^2}\Dl{\arr t{r-1}}{r-1}{Q_k,x}}
				  {p,\alpha-\gamma_1-\gamma_3,\beta-\gamma_2-\gamma_3}
			\bigg)\\
		\le\Cn
			\big(
			  1
			  +|v|^{2(\gamma_1+\gamma_2)}2^{2k(\gamma_1+\gamma_2)}
			  +|v|^{2\gamma_3}2^{2k\gamma_3}
			  +|v|^{2(\gamma_1+\gamma_2+\gamma_3)}
				2^{2k(\gamma_1+\gamma_2+\gamma_3)}
			\big)\\
		\times\norm{\Si x
			  \frac{d^2}{dx^2}\Dl{\arr t{r-1}}{r-1}{Q_k,x}}\\
		\le\Cn
			\normpar{\frac{d^2}{dx^2}\Dl{\arr t{r-1}}{r-1}{Q_k,x}}
			  {p,\alpha+1,\beta+1}.
	\end{multline*}
	
	Similarly, applying first Lemma~\ref{lm:inequality},
	then Lemma~\ref{lm:bernshtein-markov},
	and finally inequality~\eqref{eq:delta}
	we obtain
	\begin{displaymath}
		\norm{\frac1{1-x^2}\int_{-1}^1B_2(R_v)\frac{dz}{\sqrt{1-z^2}}}
		\le\Cn\norm{\frac d{dx}\Dl{\arr t{r-1}}{r-1}{Q_k,x}}
	\end{displaymath}
	and
	\begin{displaymath}
		\norm{\frac1{1-x^2}\int_{-1}^1B_3(R_v)\frac{dz}{\sqrt{1-z^2}}}
		\le\Cn\normpar{\Dl{\arr t{r-1}}{r-1}{Q_k,x}}{p,\alpha-1,\beta-1}.
	\end{displaymath}
	
	Now,
	from inequality~\eqref{eq:sumBmu}
	and the fact that $|t_r|\le\delta$
	we derive the estimate
	\begin{multline*}
		I_k\le\Cn\delta^2
		  \bigg(
			  \normpar{\frac{d^2}{dx^2}\Dl{\arr t{r-1}}{r-1}{Q_k,x}}
				{p,\alpha+1,\beta+1}\\
			  +\norm{\frac d{dx}\Dl{\arr t{r-1}}{r-1}{Q_k,x}}
			  +\normpar{\Dl{\arr t{r-1}}{r-1}{Q_k,x}}
				{p,\alpha-1,\beta-1}
			\bigg).
	\end{multline*}
	Applying twice Lemma~\ref{lm:bernshtein-markov}
	we obtain the recurrence relation
	\begin{displaymath}
		I_k=\norm{\Dl{\arr t r}r{Q_k,x}}
		\le\lastC\delta^2 2^{2k}\norm{\Dl{\arr t{r-1}}{r-1}{Q_k,x}},
	\end{displaymath}
	which yields
	\begin{displaymath}
		I_k\le\Cn\delta^{2r}2^{2kr}\norm{Q_k}
		\le\Cn M\delta^{2r}2^{k(2r-\lambda)}.
	\end{displaymath}
	
	Inequality~\eqref{eq:Ik} is proved.
	
	Now combining~\eqref{eq:J}, \eqref{eq:Ik}
	and~\eqref{eq:delta}
	we finally get
	\begin{multline*}
		I\le\Cn M
			\prn{\delta^\lambda
			  +\delta^{2r}\sum_{k=1}^N 2^{k(2r-\lambda)}}
		  \le\Cn M\prn{\delta^\lambda+\delta^{2r}2^{N(2r-\lambda)}}\\
		\le\Cn M\delta^\lambda.
	\end{multline*}
	
	The proof of Theorem~\ref{th:EsubH} is completed.
\end{proof}

\bibliographystyle{amsplain}
\bibliography{maths}

\end{document}